\documentclass[11pt,oneside,english]{amsart}
\usepackage[T1]{fontenc}
\usepackage[utf8]{inputenc}
\usepackage{color}
\usepackage{babel}
\usepackage{amstext}
\usepackage{amsthm}
\usepackage{amssymb}
\usepackage{graphicx}
\usepackage[caption=false]{subfig}
\usepackage[unicode=true,pdfusetitle,
 bookmarks=true,bookmarksnumbered=false,bookmarksopen=false,
 breaklinks=false,pdfborder={0 0 0},backref=false,colorlinks=true]
 {hyperref}
\hypersetup{
 citecolor=blue}

\makeatletter
\numberwithin{equation}{section}
\numberwithin{figure}{section}
\theoremstyle{plain}
\newtheorem*{cor*}{\protect\corollaryname}
\theoremstyle{plain}
\newtheorem{thm}{\protect\theoremname}[section]
\theoremstyle{definition}
\newtheorem{defn}[thm]{\protect\definitionname}
\theoremstyle{question}

\theoremstyle{remark}
\newtheorem{rem}[thm]{\protect\remarkname}
\theoremstyle{plain}
\newtheorem{prop}[thm]{\protect\propositionname}
\theoremstyle{plain}
\newtheorem{lem}[thm]{\protect\lemmaname}
\theoremstyle{plain}
\newtheorem{cor}[thm]{\protect\corollaryname}

\theoremstyle{plain}
\newtheorem{question}[thm]{Question}

\usepackage{geometry}
\usepackage{amsmath}

\numberwithin{equation}{section}
\numberwithin{figure}{section}
\usepackage{enumitem}		% customizable list environments
 \let\footnote=\endnote
\@ifundefined{lettrine}{\usepackage{lettrine}}{}
\theoremstyle{definition}

\def\s{\sigma}

\def\L{\Lambda}

\def\D{\mathbb{D}}
\def\R{\mathbb{R}}

\def\N{\mathbb{N}}

\usepackage{float}


\subjclass[2000]{}

\def\s{\sigma}
\def\Si{\Sigma_k}

\DeclareMathOperator{\Grass}{Grass}

\def\D{\Delta}
\def\R{\mathbb{R}}

\def\Grass{\text{Grass}}

\def\A{\mathcal{A}}
\def\B{\mathcal{B}}

\def\D{\mathcal{D}}
\def\L{\mathcal{L}}
\def\M{\mathcal{M}}

\def\Sig{\Sigma_k}
\def\glr{\text{GL}_d(\R)}

\makeatother

  \providecommand{\corollaryname}{Corollary}
  \providecommand{\definitionname}{Definition}
  \providecommand{\lemmaname}{Lemma}
  \providecommand{\propositionname}{Proposition}
  \providecommand{\remarkname}{Remark}
  \providecommand{\theoremname}{Theorem}
\providecommand{\theoremname}{Theorem}

\usepackage{tikz,xcolor,hyperref}

\definecolor{lime}{HTML}{A6CE39}
\DeclareRobustCommand{\orcidicon}{
	\begin{tikzpicture}
	\draw[lime, fill=lime] (0,0) 
	circle [radius=0.16] 
	node[white] {{\fontfamily{qag}\selectfont \tiny ID}};
	\draw[white, fill=white] (-0.0625,0.095) 
	circle [radius=0.007];
	\end{tikzpicture}
	\hspace{-2mm}
}
\foreach \x in {A, ..., Z}{\expandafter\xdef\csname orcid\x\endcsname{\noexpand\href{https://orcid.org/\csname orcidauthor\x\endcsname}
			{\noexpand\orcidicon}}
}

\author[Reza Mohammadpour]{Reza Mohammadpour\orcidA{}}

\address{Department of Mathematics, Uppsala University, Box 480, SE-75106, Uppsala, SWEDEN.}
\date{\today}

\email{reza.mohammadpour@math.uu.se}

\begin{document}
\title[Restricted variational principle for Lyapunov exponents]{Restricted variational principle for Lyapunov exponents of typical one-step cocycles}
\maketitle
\begin{abstract}
In this paper, we study the multifractal formalism for Lyapunov exponents of typical one-step cocycles.  We establish a variational relation between the Legendre transform of topological pressure of the generalized singular value function and measure-theoretic entropies. As a consequence, we show that the restricted variational principle of Lyapunov exponents holds for typical one-step cocycles. Furthermore, we prove a variational principle  formula for the generalized singular value function. 
\end{abstract}

\section{Introduction and statement of the results}
 A continuous map $T: X\rightarrow X$ on the compact metric space $X$ is called topological dynamical system and we denote it by $(X,T)$. A \textit{matrix cocycle} $\A$ on a compact metric space $X$ is a continuous map $\A \colon X \to \glr$ over a topological dynamical system $(X, T)$. For $n\in \N$ and $x \in X$, we define the product of $\A$ along the length $n$ orbit of $X$ as   
$$\A^n(x):= \A(T^{n-1}(x)) \ldots \A(x).$$

Let $k \in \N$ be given.
The one-sided shift $\Si$ of $k$ symbols is a space $\{1,2,\ldots,k\}^\N$. Let $T: \Si \to \Si$ be the left shift map. For simplicity, we denote it by $(\Si, T)$. For each $n\in\N$, we define $\L_n:=\{1,\ldots,k\}^n$ to be the set of all length $n$ words of $\Si$.  Also, we denote by $\M(\Si, T)$ the space of all $T$-invariant Borel probability measures
on $\Si$. 

\color{black}

 A well-known class of matrix cocycles is \textit{one-step cocyles}, which is defined as follows.  Suppose that $T:\Si \rightarrow \Si$  is a left shift map. Given a $k$-tuple of matrices $\textbf{A}=(A_{1},\ldots,A_{k})\in \glr^{k}$, we associate with it the locally constant map $\mathcal{A}:\Si \rightarrow \glr$ given by $\mathcal{A}(x)=A_{x_{0}},$ that means the matrix cocycle $\mathcal{A}$ depends only on the zero-th symbol $x_0$ of $(x_{l})_{l\in \N}$. In this case, we say that $\mathcal{A}$ is a one-step cocycle. The $k$-tuple of matrices $\textbf{A}$ is called the generator of the cocycle $\A$. For any length $n$ word $I=i_{0}, \ldots, i_{n-1},$  we denote 
\[\mathcal{A}_{I}:=A_{i_{n-1}}\ldots A_{i_{0}}.\]

Therefore, when $\A$ is a one-step cocycle, \[ \A^n(x)=A_{x_{n-1}}\ldots A_{x_{0}}.\] 
 For each
$T$-invariant probability measure $\mu$ such that $\log\|\mathcal A^{\pm 1}\|\in L^1(\mu)$, Oseledets' theorem ensures that for $\mu$-almost every $x$ there exist $k(x)\ge 1$ and an $\mathcal A$-invariant splitting $\mathbb R^d=E^1_x \oplus E^2_x \oplus \dots \oplus E_x^{k(x)}$ so that the \emph{Lyapunov exponents} are well defined by the limits
\begin{equation*}
    \lambda_i(\mathcal A,x)=\lim_{n\to\infty} \frac1n \log \|\mathcal A^n(x)v\|, \quad \text{for every $v\in E^i_x\setminus\{0\}$ }.
\end{equation*}
We denote by 
$ 
\chi_{1}(\mathcal{A}, x ) \geq \chi_{2}(\mathcal{A}, x )\geq \ldots \geq \chi_{d}(\mathcal{A}, x )
$  
the Lyapunov exponents, counted with multiplicity, of the cocycle $\mathcal A$ at $x\in \Sigma$. Set also 
$\chi_i(\mu, \A):=\int \chi_i(\mathcal{A}, \cdot ) d\mu $  for $1\le i \le d.$ We denote  
\[ \Omega:=\left\{(\chi_{1}(\mu, \mathcal{A}), \chi_{2}(\mu, \mathcal{A}),...,\chi_{d}(\mu, \mathcal{A})): \mu \in \M(\Sig, T)\right\}.\]
We note that the Lyapunov exponents $\chi_i(\mathcal A,x)$ can be computed as 
$$
\chi_i(\mathcal A,x)
= \lim_{n\to \infty}\frac{1}{n}\log \sigma_{i}(\mathcal{A}^{n}(x))
$$
where $\sigma_{1}, \ldots, \sigma_d$ are the singular values, listed in decreasing order according to multiplicity.
 The \textit{Lyapunov spectrum} of the matrix cocycle 
$\mathcal A$ is the set
\begin{equation}
\label{def:Lyapunov-spectrum}
\vec{L}=\bigg\{\vec{\alpha} \in \R^{d}: 
E(\vec{\alpha}) \neq \emptyset \bigg\}  \end{equation}
where,  for each
$\vec{\alpha}:=(\alpha_1, \ldots, \alpha_d) \in \R^{d}$, the $\vec{\alpha}$-\textit{level set} is
 defined by 
 \begin{equation}\label{eq:123defl}
E(\vec{\alpha})=\bigg\{ x\in \Sig: \lim_{n\to \infty}\frac{1}{n}\log \sigma_{i}(\mathcal{A}^{n}(x))=\alpha_i \text{  for every } 1\le i \le d \bigg\}.   
 \end{equation}
 \color{black}

When $\mu$ is an ergodic invariant probability measure, $\chi_{i}(\mathcal{A}, x)$ is constant for $\mu$-almost every $x$, for each $i = 1, \ldots, d$. Since there are many ergodic invariant measures, the Lyapunov exponents may take different values. Therefore, it is natural to study the size of $E(\vec{\alpha})$. This size is typically measured either by the Hausdorff dimension or by the topological entropy (in the sense of \cite{Bowen}). We denote by $h_{\text{top}}(E(\vec{\alpha}))$ the topological entropy of the $\vec{\alpha}$-level set in the sense of \cite{Bowen}.

\color{black}
Feng and Huang \cite{FH} calculated $h_{\text{top}}(E(\vec{\alpha}))$ for special classes of linear cocycles so that
the natural family of potentials is almost additive (to be defined in Subsection \ref{subsec:subadd}), extending the work of Barreira and Gelfert \cite{BG06}. It is important to note that the almost additivity condition is a strong assumption and applies only to a restrictive set of matrix cocycles.

In this paper, we establish a variational relation between the Legendre transform of the topological pressure of the generalized singular value function and measure-theoretic entropies for generic one-step matrix cocycles. As a consequence, we describe $h_{\text{top}}(E(\vec{\alpha}))$ in terms of the measure-theoretic entropy with a vector containing all the Lyapunov exponents (with respect to the same measure) for generic one-step matrix cocycles. This formula is called the {\it restricted variational principle} (see Corollary \ref{corollary}).

While there are several results on the entropy spectrum of the top (or largest) Lyapunov exponent (e.g., see \cite{feng03, feng09, FH, BARREIRA20091, DGR19, Moh22-Lyapunov}), there are few results that calculate the entropy spectrum of a vector containing all Lyapunov exponents. The relatively small body of research on the entropy spectrum of all Lyapunov exponents can perhaps be explained by the fact that the behavior of the vector of all Lyapunov exponents in matrix cocycles is much more complicated than that of the top Lyapunov exponent. We refer the reader to \cite{climenhaga, Barreira-survey, Mohammadpour-survey} for references and full details.

 %We also define the \textit{Lyapunov spectrum}

%\[\vec{L}=\bigg\{\vec{\alpha} \in \R^{d}: \exists x \in \Sigma \text{ such that } \lim_{n \to \infty} \frac{1}{n}\log \sigma_{i}(\mathcal{A}^{n}(x))=\alpha_i \bigg\}.\]

 For $q:=(q_{1}, \cdots,  q_{d})\in \R^d$, we define the \textit{generalized singular value function} $\psi^{q_{1}, \ldots, q_{d}}(\mathcal{A}): \mathbb{R}^{d \times d} \rightarrow[0, \infty)$ as
\begin{equation}\label{eq:singularvfunctioncocycle}
\psi^{q_{1}, \ldots, q_{d}}(\mathcal{A}):=\sigma_{1}(\mathcal{A})^{q_{1}} \cdots \sigma_{d}(\mathcal{A})^{q_{d}}.
\end{equation}

  For any $q:=(q_{1}, \cdots,  q_{d})\in \R^d$, we denote $\psi^{q}(\A):=\psi^{q_{1}, \ldots, q_{d}}(\mathcal{A}).$ There is some similarity between Falconer’s singular value function $\varphi^{s}(\mathcal{A})$ (see \cite{falconer}) and the generalized singular value function $\psi^{q}(\mathcal{A})$, as when $s \in [0, d]$, the singular value function $\varphi^{s}(\mathcal{A}(\cdot))$ coincides with the generalized singular value function $\psi^{q_{1}, \ldots, q_{d}}(\mathcal{A}(\cdot))$, where
\[
\left(q_{1}, \ldots, q_{d}\right)=(\underbrace{1, \ldots, 1}_{m \text { times }}, s-m, 0, \ldots, 0),
\]
and $m=\lfloor s \rfloor$. We should note that while Falconer's singular value function $\varphi^{s}(\mathcal{A})$ is submultiplicative, the generalized singular value function $\psi^{q}(\mathcal{A})$ is neither submultiplicative nor supermultiplicative.

\color{black}
 In this paper, we consider \textit{typical one-step cocycles} (see Section \ref{section-prelim} for the definition). The assumptions of typicality  are considered as suitable generalizations of the concepts of proximality and strong irreducibility within the context of random products of matrices (cf. \cite{Viana-book, Furstenberg1963}). Bonatti and Viana \cite{bonatti2004lyapunov} introduced the concept of typical cocycles, showing that the set of typical cocycles is both open and dense within the set of fiber-bunched cocycles, and that its complement has infinite codimension. Note that they require the fiber-bunching assumption as one of the key ingredients for the existence of stable and unstable holonomies, which is guaranteed in the context of one-step cocycles (see \cite{Moh22-Lyapunov, bonatti2004lyapunov}).

The \emph{topological pressure} of the generalized singular value potential~\eqref{eq:singularvfunctioncocycle} is defined by 
\begin{equation}\label{eq:defPressure}
P^*( \log \psi^q(\mathcal{A}) ):=\limsup_{n \to \infty}\frac{1}{n} \log \sum_{I \in \L_n} \psi^q(\mathcal{A}_{I}),     \hspace{0.5cm} \forall q \in\R^d.    
\end{equation}
In case $\mathcal{A}$ is a typical cocycle, one can replace $\limsup$ by a limit in ~\eqref{eq:defPressure} and we denote it by $P(\log \psi^q(\mathcal{A}))$ (see Lemma \ref{existence-of-top-pre}). 
 
 The measure-theoretic entropy of $T$ with respect to $\mu \in \M(\Si, T)$ exists, being defined by
\[
h_{\mu}(T):=\lim_{n \to \infty}\frac{-1}{n}\sum_{I \in \mathcal{L}_n} \mu([I])\log \mu([I]),
\]
where $[I]$ is a cylinder set and  $0\log 0=0.$ The measure $\mu$ is also invariant under $T^{m}$ and we have $h_{\mu}(T^{m})=mh_{\mu}(T)$ for each $m \in\N$. For more information on measure-theoretic entropy, we refer the reader to \cite{PU}.
\color{black}

We denote by $\text{ri}(S)$ the relative interior of a set $S$ (cf. \cite{Rock}). Our first result is as follows:

\begin{thm}\label{main-thm}
Assume that $(A_1, \ldots, A_k)\in \glr^k$ generates a one-step cocycle $\A:\Si \to \glr$. Suppose that $\A:\Si \to \glr$  is a typical cocycle. 
 Then,
\[
\begin{aligned}
&\sup\bigg\{h_{\mu}(T) : \mu \in \mathcal{M}(\Si, T), \chi_{i}(\mu, \mathcal{A})
=\alpha_{i} \text{ for }i=1,2, \ldots, d \bigg\}\\
&=\inf_{q\in \R^d} \bigg\{P(\log \psi^{q}(\A))-\langle q, \vec{\alpha} \rangle \bigg\}
\end{aligned}\]
for $\vec{\alpha} \in \mathrm{ri}(\Omega)$.
\end{thm}

In the following corollary, we show that the restricted variational principle of Lyapunov exponents holds for typical one-step cocycles.

\begin{cor}\label{corollary}
Assume that $(A_1, \ldots, A_k)\in \glr^k$ generates a one-step cocycle $\A:\Si \to \glr$. Suppose that $\A:\Si \to \glr$  is a typical cocycle. Then
$$\begin{aligned}
h_{\text{top}}(E(\vec{\alpha}))=& \inf_{q \in \R^d}\bigg\{P(\log \psi^{q}(\A))- \langle \vec{\alpha}, q \rangle \bigg\}\\
&=\sup \bigg\{h_{\mu}(T): \mu \in \M(\Si, T), \chi_{i}(\mu, \A)=\alpha_i \text{ for }i=1,2, \ldots, d \bigg\}\\
\end{aligned}$$
for all $\vec{\alpha} \in \mathrm{ri}(\Omega).$
\end{cor}
\begin{proof}
It follows from the combination of \cite[Theorem A]{Moh22-entropy} and Theorem \ref{main-thm}.
\end{proof}

Cao, Feng and Huang \cite{CFH08} proved a variational principle formula for \color{black} sub-additive potentials. \color{black} A variational principle formula for super-additive potentials was proved in Cao, Pesin and Zhao \cite{CPZ}.  $\psi^q(\mathcal{A})$ is neither super-multiplicative nor sub-multiplicative. Therefore, one can not use their results to prove a variational principle formula for $\log  \psi^q(\mathcal{A})$ for any $q \in \R^d.$ Notice that the limit in defining the Lyapunov exponent \[\lim_{n\to \infty} \frac{1}{n} \int \log \psi^{q}(\A^n(x)) d\mu(x)\] exists with respect to any invariant measures $\mu \in \M(\Si, T)$ for any $q \in \R^{d}$ (see Remark \ref{LE exists}).

In the following result, we prove a variational principle formula for $\log  \psi^q(\mathcal{A})$ for any $q \in \R^d.$

\begin{thm}\label{vair-prin-for the generalized sing}
Assume that $(A_1, \ldots, A_k)\in \glr^k$ generates a one-step cocycle $\A: \Si \to \glr$. Let $\A:\Si \to \glr$ be a typical cocycle.  Then, a variational principle  holds for $\{\log \psi^q(\A^n)\}_{n\in \N}$ for any $q \in \R^d$; that is
\[ P(\log \psi^{q}(\A))=\sup\bigg\{ h_{\mu}(T)+\lim_{n\to \infty} \frac{1}{n} \int \log \psi^{q}(\A^n(x)) d\mu(x): \mu \in \M(\Si, T) \bigg\} \]
for any $q \in \R^d.$
\end{thm}

In Section \ref{section-prelim} we introduce some notation and preliminaries. In Section \ref{proof-of-the-VR} we prove  Theorem \ref{vair-prin-for the generalized sing}. We prove Theorem \ref{main-thm} in Section \ref{proof-of-the-main-thm}.

\section{preliminaries}\label{section-prelim}
For each $n\in\N$, we recall that $\L_n$  is the set of all length $n$ words of $\Si$. We define $\L:=\bigcup\limits_{n\in \N}\L_n$ to be the set of all words. If $i\in \L$, then we define $\left.i\right|_{n}=i_{0} \cdots i_{n-1}$ for all $n \in \mathbb{N}$. The empty word $\left.i\right|_{0}$ is denoted by $\varnothing$. The length of $i \in \L$ is denoted by $|i|$. The longest common prefix of $i, j \in \L \cup \Si $ is denoted by $i \wedge j$. The concatenation of two words $i \in \L \cup \Si$ and $j \in \L $ is denoted by $ji$.  If $i \in \L_n$ for some $n$, then we set $[i]=\left\{j \in \Si:\left.j\right|_{n}=i\right\}$. The set $[i]$ is called a \textit{cylinder set}. A cylinder containing $x=\left(x_{i}\right)_{i \in \mathbb{Z}} \in \Si$ of length $n \in \mathbb{N}$ is defined by
$[x]_{n}:=\left\{\left(y_{i}\right)_{i \in \mathbb{N}} \in \Si : x_{i}=y_{i} \text { for all } 0 \leq i \leq n-1 \right\}.$  The shift space $\Si$ is compact in the topology
generated by the cylinder sets. Moreover, the cylinder sets are open and closed in
this topology and they generate the Borel $\sigma$-algebra.

\color{black}Let $\mu \in \M(\Sig, T)$ and let $Z \subset \Sig$ be a measurable set with $\mu(Z)>0$. \color{black} By Poincaré's recurrence theorem, for almost every $x \in Z$ the return time $
n_{Z}(x)=\min\{n \geq 1 : T^n(x)\in Z\}$
is finite. This defines the induced map $T_Z : Z \to Z$ by
$T_Z(x)=T^{n_{Z}(x)}(x).$
The normalized measure $\mu_Z(B) =\frac{\mu(B)}{\mu(Z)}$ is $T_Z$-invariant, and the system $(Z, T_Z, \mu_Z)$ is called the \emph{induced system.}

\color{black}

\subsection{Typical cocycles}\label{typical cocycles}

We
denote by $\Grass(k, d)$ the $k$-th Grassmannian, i.e., the set of all $k$-dimensional
subspaces of $\R^d$.

 The eccentricity of a linear map $B \in \glr$ is defined by 
\[\operatorname{Ecc}(B)=\min _{1 \leq i<d} \frac{\sigma_{i}(B)}{\sigma_{i+1}(B)}.\]

\begin{defn}\label{typical1}
Suppose that $\mathcal{A}:\Si \rightarrow \glr$ is a one-step cocycle generated by  $(A_{1},\ldots,A_{k}) \in \glr^k$. We say that $\mathcal{A}$ is \textit{typical} if 
\begin{itemize}
\item[(i)]$\{\A_I \}_{I \in \L}$ contains elements with arbitrarily large
eccentricity;
\item[(ii)] for any $1\leq i \leq d-1,$ and any $F\in \Grass(i,d)$, and any finite family $G_{1},\ldots,G_{n} \in \Grass(d-i,d),$ there exists $K \in \L$ such that $\A_{K}\left(F\right) \cap G_{j}=\{0\}$ for
every $1\leq j \leq n.$
\end{itemize}

\end{defn}
The above definition comes from \cite[Section 8]{Viana-book} and \cite[Subsection A.4.5]{AV07}. Avila and Viana \cite[Subsection A.4.5]{AV07} showed that this definition is equivalent to the general definitions of typical cocycles given in \cite[Definition 1.2]{AV07}. Avila and Viana \cite{AV07, AV07-acta}, as well as Bonatti and Viana \cite{bonatti2004lyapunov}, showed that the set of typical cocycles is open and dense. Their works are considered an extension of Furstenberg's work \cite{Furstenberg1963}, in which he proved the similarity of Lyapunov exponents under the assumptions of proximality and strong irreducibility (see \cite{Viana2020, Viana-book} for more details).

One can think of typicality as being an analogue of strong irreducibility and 
proximality for cocycles on subshifts of finite type. In fact, it is not difficult to show
that one-step {\color{black}$\text{GL}_{2}(\R)$}-cocycles which are proximal and strongly irreducible
are typical.

\subsection{Sub-additive thermodynamic formalism}\label{subsec:subadd}

 Let $\Phi=\{\log \phi_n\}_{n \in \N}$ be a \textit{sub-additive potential} over a full shift $(\Sig, T)$, i.e., each $\phi_{n}$ is a continuous positive-valued function on $X$ such that
\[ 0<\phi_{n+m}(x) \leq \phi_{n}(x) \phi_{m}(T^{n}(x)) \quad \forall x\in \Sig, m,n \in \N.\]

Similarly, we call a sequence of continuous functions (potentials) $\Phi=\{\log \phi_n\}_{n \in \N}$
\textit{super-additive} if $-\Phi=\{-\log \phi_n\}_{n \in \N}$ is sub-additive.

Moreover,  $\Phi=\{\log\phi_{n}\}_{n=1}^{\infty}$ is said to be an \textit{almost additive potential} if there exists a constant $C > 0$ such that for any $m,n \in \N$, $x\in \Sig$, we have
\[
C^{-1}\phi_{n}(x)\phi_{m}(T^{n})(x) \leq \phi_{n+m}(x)\leq C \phi_{n}(x) \phi_{m}(T^{n}(x)).
\]

 The sub-additive variational principle (see \cite{CFH08}) states that
\begin{equation}\label{varitional-principle}
P\left(\Phi\right)=\sup \bigg\{h_{\mu}(T)+\lim _{n \rightarrow \infty} \frac{1}{n} \int \log \phi_{n}(x) d \mu(x) :   \mu \in \mathcal{M}(\Sig, T) \bigg\},
\end{equation}
where $P(\Phi)$ is the topological pressure of the sub-additive potential $\Phi$ and $h_{\mu}(T)$ is the \textit{measure-theoretic entropy}. Moreover, the super-additive variational principle was proved in \cite{CPZ}.

Let $\A:\Sig \to \glr$ be a matrix cocycle over a full shift $(\Sig, T)$. The submultiplicativity of the norm $\|\cdot\|$ implies that  $\| \A \|$ is submultiplicative in the sense that for any $m, n \in \mathbb{N}$, and $x \in \Sig$,
$$
0 < \|\mathcal{A}^{n+m}(x)\| \leq \|\mathcal{A}^{n}(T^{m}(x))\| \|\mathcal{A}^{m}(x)\|.
$$
This gives rise to a norm potential $\left\{\log \| \A^{n}\| \right\}_{n \in \mathbb{N}}$. Therefore, we can use the variational principle for topological pressure for the potential $\left\{\log \| \A^{n}\| \right\}_{n \in \mathbb{N}}.$

\color{black}
Let $\mu \in \mathcal{M}(\Sig,T)$ and let $Z \subset \Sig$ be a measurable set with $\mu(Z)>0$. \color{black} We define the \emph{induced cocycle} by $
\mathcal{A}_Z(x)=\mathcal{A}^{n_Z(x)}(x)$ over the induced system $(Z, T_{Z})$.
Wojtkowski \cite{Woj} showed that
\begin{equation}\label{Wojtkowsky}
\chi_1(\mu_{Z}, \mathcal{A}_Z)
= \left( \int_Z n_Z(x)\, d\mu_Z(x) \right)\, \chi_1(\mu, \mathcal{A}).
\end{equation}

\color{black}

Assume that $(A_1, \ldots, A_k)\in \glr^k$ generates a one-step cocycle $\A:\Si \to \glr.$ We denote $\Phi_{\A^n}:=(\log \s_1(\A^n), \ldots, \log \s_d(\A^n))$ for any $n \in \N$. For any $n \in \N$ and $q \in \R^d$, $\langle q, \Phi_{\A^n} \rangle=\log \psi^{q}(\A^n).$

We will also make use of the exterior product cocycle $\A^{\wedge m}$ for $1 \leq m \leq  d-1$ where
$\A^{\wedge m}(x)$ is considered as a linear transformation on $(\R^{d})^{\wedge m}$. \color{black}Observe that exterior powers commute with iteration and induction, namely
\begin{equation}\label{commute}
(\A^n(x))^{\wedge m}=(\A^{\wedge m})^n(x)
\end{equation}
for all $n \in \N$.
\color{black}

\begin{rem}\label{LE exists} For any $q=(q_1, \ldots, q_d) \in \R^d$, we can write
\begin{equation}\label{relation of psi and product}
\psi^{q}(\A)=\prod_{i=1}^{d} \|\A^{\wedge i}\|^{t_i},
\end{equation}
where $t_i=q_{i}-q_{i+1}$ and $q_{d+1}=0$ for $i \in \{1, \ldots, d\}$. Since $\|\A^{\wedge i}\|^{t_i}$ is either sub-multiplicative or super-multiplicative, $\lim_{n\to \infty}\frac{1}{n} \int \log \|\A^{\wedge i}_{x_{|n}}\|^{t_i} d\mu(x)$ exists for any invariant measure $\mu \in \M(\Si, T)$. Thus, 
\[\lim_{n\to \infty}\frac{1}{n} \int \log \psi^{q}(\A_{x_{|n}}) d\mu(x)\]
exists. \color{black} Moreover,  for $Z \subset \Sig$ a measurable set with $\mu(Z)>0$, we have
\[\begin{aligned}
&\lim_{n\to \infty}\frac{1}{n} \int \log \psi^{q}((\A_{Z})^n(x)) d\mu_Z(x)\\
&\stackrel{\eqref{relation of psi and product}\text { and } \eqref{commute}}{=}\sum_{i=1}^{d}t_i \chi_{1}(\mu_Z, \A_{Z}^{\wedge i})\\
& \stackrel{\eqref{Wojtkowsky}}{=}\left( \int_Z n_Z(x)\, d\mu_Z(x) \right)\sum_{i=1}^{d}t_i \chi_{1}(\mu, \A^{\wedge i})\\
&=\left( \int_Z n_Z(x)\, d\mu_Z(x) \right) \lim_{n\to \infty}\frac{1}{n} \int \log \psi^{q}(\A^{n}(x)) d\mu(x).
\end{aligned}
\]
\color{black}
\end{rem}

\section{Proof of Theorem \ref{vair-prin-for the generalized sing}}\label{proof-of-the-VR}

We recall that   $\psi^q(\mathcal{A})$ is neither super-multiplicative nor sub-multiplicative. \color{black} For one-step cocycles, the limsup topological pressure of $\log \psi^q(\mathcal{A})$ is defined in \eqref{eq:defPressure}. 
When the limit exists, we denote the topological pressure by $P(\log \psi^q(\mathcal{A})).$  \color{black}
\begin{lem}\label{existence-of-top-pre}
Assume that $(A_1, \ldots, A_k) \in \glr^k$ generates a one-step cocycle $\A: \Si \to \glr.$ Let $\A:\Si \to \glr$ be a typical cocycle.  Then the limit in defining  $P^*(\log \psi^q(\mathcal{A}))$ exists for any $q\in \R^{d}$. Moreover, $P(\log \psi^q(\mathcal{A}))$ is a
convex function of $q$ in $\R^d$.
\end{lem}
\begin{proof}
By \cite[Theorem 4.1]{Park20}, typical cocycles are simultaneously quasi-multiplicative (see \cite[Remark 3.4]{Moh22-entropy}). Hence, we immediately obtain the existence of the limit by \cite[Lemma 3.1]{Moh22-entropy}. Moreover, the convexity of $P(\log\psi^{q}(\A))$ follows by a standard argument (cf. \cite{PU}).
\end{proof}

\begin{thm}\label{upper-bound of var-pri}
Assume that $(A_1, \ldots, A_k) \in \glr^k$ generates a one-step cocycle $\A: \Si \to \glr.$ Let $\A:\Si \to \glr$ be a typical cocycle. Then,
\[
\sup\bigg\{ h_{\mu}(T)+\lim_{n\to \infty} \frac{1}{n} \int \log \psi^{q}(\A^n(x)) d\mu(x): \mu \in \M(\Si, T) \bigg\} \leq P(\log \psi^{q}(\mathcal{A}))
\]
for any $q \in \R^d.$
\end{thm}
\begin{proof}
 We recall the following inequality (see \cite{Bow}) that we use for the proof:   \[\sum_{i=1}^m p_i\left(c_i-\log p_i\right) \leq \log \sum_{i=1}^m e^{c_i},\] where $c_i \in \mathbb{R}, p_i \geq 0$ and $\sum_{i=1}^m p_i=1$.

Assume that $\mu \in \M(\Si, T).$ Therefore, $\sum_{I \in \mathcal{L}_n} \mu([I])=1$ . Thus, by using the above inequality,

\[ \frac{1}{n}\sum_{I \in \mathcal{L}_{n}} \mu([I]) \big(-\log \mu([I])+\log \psi^q(\A_I)) \leq \frac{1}{n} \log \sum_{I \in \mathcal{L}_n}  \psi^q(\A_I),\]
 for any $q \in \R^d.$

By Lemma \ref{existence-of-top-pre} and Remark \ref{LE exists}, the limits defining $\lim_{n \to \infty} \frac{1}{n} \int \log \psi^{q}(\mathcal{A}^n(x)) d\mu(x)$ and $P(\log \psi^q(\mathcal{A}))$ exist for any $q \in \mathbb{R}^{d}$, respectively. Letting $n \to \infty$ gives
\[ h_{\mu}(T)+\lim_{n \to \infty} \int \frac{\log \psi^q(\A_{x_{|n}})}{n} d\mu(x) \leq P(\log \psi^q(\A)).\]

 Therefore,
\[
\sup\bigg\{ h_{\mu}(T)+\lim_{n\to \infty} \frac{1}{n} \int \log \psi^{q}(\A^n(x)) d\mu(x): \mu \in \M(\Si, T) \bigg\} \leq P(\log \psi^{q}(\mathcal{A})).
\]

\end{proof}
Let $\textbf{A}$ be a compact set in $\glr.$ We say that $\textbf{A}$ is \textit{dominated} of index  $i$ iff there exist $C>0$  and $0<\tau<1$ such that for any finite sequence $A_1, \ldots, A_N$ in $\textbf{A}$ we have
$$
\frac{\sigma_{i+1}\left(A_1 \cdots A_N\right)}{\sigma_i\left(A_1 \cdots A_N\right)}<C \tau^N.
$$

 A one step cocycle $\A$ generated by $\textbf{A}$ is dominated of index  $i$ if $\textbf{A}$ is dominated of index  $i$ (see \cite{BG}).
We also say that $\textbf{A}$ is dominated iff it is dominated of index  $i$ for each $i\in\{1, \ldots, d-1\}.$ Similarly, a one step cocycle $\A$ generated by $\textbf{A}$ is dominated if $\textbf{A}$ is dominated (see \cite{BG}).

 \begin{thm}[{\cite[Corollary 4.6]{Moh22-entropy}}] \label{dominated-one-step-cocycle}
Assume that $(A_1, \ldots, A_k) \in \glr^k$ generates a one-step cocycle $\A: \Si \to \glr.$  Suppose that $\mathcal{A}:\Si \to \glr$ is a typical cocycle. Then, there exists $K_0 \in \N$ such that for every $n\in \N$ and $I \in \mathcal{L}_{n}$ there exist $J_2=J_2(I)$ and $J_1=J_1(I)$ with $|J_{i}|\leq K_0$ for $i=1,2$ such that the tuple \[\left(\mathcal{A}_{\mathrm{k}}\right)_{\mathrm{k} \in \mathcal{L}_{n}^{\mathcal{D}}}, \quad \text{where } \mathcal{L}_{n}^{\mathcal{D}}:=\left\{J_1(I) IJ_2(I): I \in \mathcal{L}_{n}\right\},\]
is dominated.
\end{thm}
Note that for each $I \in \mathcal{L}_{n}^{\mathcal{D}}$, $|I| \in [n, n+2K_0].$   Let $\A: \Si \to \glr$ be a one-step cocycle.   Assume that $\A:\Si \to \glr$ is a typical cocycle. By Theorem \ref{dominated-one-step-cocycle}, the one-step cocycle $\mathcal{B}: (\mathcal{L}_{n}^{\mathcal{D}})^{\N} \to \glr$ over a full shift $((\mathcal{L}_{n}^{\mathcal{D}})^{\N}, f )$ defined by $\mathcal{B}(\omega):=\A_{J_1(I)IJ_2(I)}$, where $\mathcal{B}$ depends only on the zero-th symbol $J_1(I)IJ_2(I)$ of $\omega \in (\mathcal{L}_{n}^{\mathcal{D}})^{\N}$, is dominated. It is easy to see that $(\mathcal{L}_{n}^{\mathcal{D}})^{\N} \subset \Si$. We denote 
\[
\Psi(\mathcal{B})=\left(\log \sigma_{1}(\mathcal{B}), \ldots, \log \sigma_{d}(\mathcal{B})\right).
\]

We define a pressure on the dominated subsystem $\mathcal{L}_{n}^{\mathcal{D}}$ by setting \color{black}
\[P_{n, \mathcal{D}}(\log\varphi):=\lim _{m \rightarrow \infty} \frac{1}{m} \log \sum_{I_{1}, \ldots, I_{m} \in \mathcal{L}_{n}^{\mathcal{D}}}\varphi(I_1 \ldots I_m),\]
\color{black}whenever $\varphi: \mathcal{L} \rightarrow \mathbb{R}_{\geq 0}$ is sub-multiplicative, i.e.,
$$
\varphi(\mathrm{I}) \varphi(\mathrm{J}) \geq \varphi(\mathrm{IJ}).
$$
for all $I, J \in \mathcal{L}$ with $IJ \in \mathcal{L}.$

Let $\mathfrak{A}: \Si \to \glr$ be a one-step cocycle. If the cocycle $\mathfrak{A}$ is dominated of index  $i$, then $\{\log \sigma_{i}(\mathfrak{A}^n)\}_{n\in \N}$ is almost additive by \cite[Proposition 5.8]{Moh22-Lyapunov} (see also \cite[Remark 4.2]{Moh22-entropy}).  Note that the domination property can also be characterized in terms of existence of invariant cone fields
(see \cite[Theorem B]{BG}).

\begin{lem}\label{psi is almost additive}
Assume that $(A_1, \ldots, A_k) \in \glr^k$ generates a one-step cocycle $\A: \Si \to \glr.$  Suppose that $\mathcal{A}:\Si \to \glr$ is a typical cocycle. Then, the potential $\{\langle q, \Psi(\B^n) \rangle\}_{n \in \N}$ is almost additive for any  $q \in \R^d.$ 
\end{lem}
\begin{proof}
As we explained above, the one-step cocycle $\mathcal{B}: (\mathcal{L}_{n}^{\mathcal{D}})^{\N} \to \glr$ over a full shift $((\mathcal{L}_{n}^{\mathcal{D}})^{\N}, f )$ is dominated. Therefore, $\{\log \sigma_i(\mathcal{B}^{n})\}_{n \in \N}$ is almost additive for each $i \in \{1, \ldots, d-1\}$ by \cite[Proposition 5.8]{Moh22-Lyapunov}. Hence, $\{\langle q, \Psi(\B^n) \rangle\}_{n \in \N}$ is almost additive for any  $q \in \R^d.$ 
\end{proof}

The following proposition could be seen as a generalization of \cite[Proposition 5.2]{DGR17}, with statement for general ergodic measures as opposed to the measure of maximal entropy.

\begin{prop}\label{relation between entropies and LE}
Assume that $(A_1, \ldots, A_k) \in \glr^k$ generates a one-step cocycle $\A: \Si \to \glr.$  Suppose that $\mathcal{A}:\Si \to \glr$ is a typical cocycle.  There exists $K_0 \ge 1$ so that, for each $n>2K_0$ there is a dominated one-step cocycle 
$\mathcal{B}: (\L_{n}^{\D})^{\N}\to \glr$ over a compact full shift $((\L_{n}^{\D})^{\N}, f)$ (depending on $n$)  satisfying the following:

For any $\mu' \in \M((\L_{n}^{\D})^{\N}, f)$, there is $\mu \in \M(\Si, T)$ such that
\begin{equation}\label{entropies-relation}
 h_{\mu'}(f) \leq (n+2K_0)h_{\mu}(T)+\frac {n+2K_0}n \log (2K_0+1),
\end{equation}
and 
\color{black}
\begin{equation}\label{LE-relations}
\lim_{m \to \infty} \frac{1}{m} \int \langle q, \Psi(\B^m(x)) \rangle d\mu'(x) \leq (n+2K_{0})\lim_{m\to \infty} \frac{1}{m} \int \log \psi^{q}(\A^m(x)) d\mu(x).
\end{equation}
\color{black}
\end{prop}
\begin{proof}

There is a natural projection 
\[
\pi : (\mathcal{L}_{n}^{\mathcal D})^{\mathbb{N}} \to \Sig
\]
given by the substitution map. In this projection, each symbol $I \in \mathcal{L}_{n}^{\mathcal D}$ is mapped to a word of length in $[n, n+2K_0]$.

For any $\mu' \in \mathcal{M}((\mathcal{L}_{n}^{\mathcal D})^{\mathbb{N}}, f)$, the image $\nu = \pi_*(\mu')$ is not $T$-invariant on $\Sigma$. To make it invariant, define
\[
\mu = \frac{1}{Z} \sum_{I \in \mathcal{L}_{n}^{\mathcal D}} \sum_{i=0}^{|I|-1} T^i_* \big( \pi_*(\mu'|_{C[I]}) \big),
\]
where $C[I]$ is the cylinder in $(\mathcal{L}_{n}^{\mathcal D})^{\mathbb{N}}$ consisting of sequences whose first symbol is $I$, and
\[
Z = \sum_{I \in \mathcal{L}_{n}^{\mathcal D}} |I| \, \mu'(C[I])
\]
is the normalizing constant, which $n \leq Z \leq n+2K_0$. The measures $\mu$ thus obtained are supported in the space

\[
\Sigma' = \bigcup_{i=0}^{n+2K_0-1} T^i(\pi(  (\L_{n}^{\D})^{\N}).
\]

If this projection $\pi: (\L_{n}^{\D})^{\N}  \to \Sigma'$ were injective, or equivalently if the infinite sequences in $\Sigma'$ were uniquely decipherable as concatenations of words from $\L_{n}^{\D}$, then  the Abramov formula \cite{Ab} would give us
\[
h_{\mu'}(f) =  h_{\mu}(T)  \cdot \int n(I) d\mu'(C[I]) \leq (n+2K_0)h_{\mu}(T),
\]
\color{black} where $n(I)$ is the return time map.
\color{black}
However, those sequences are in general not uniquely decipherable, which leads to a (possible) drop of the measure-theoretic entropy. On the other hand, this  non-unique decipherability is not very strong. That is, any word $W\in \{1,\ldots,k\}^N$ of length $N\gg n$ can be presented as a concatenation of words from $ \L_{n}^{\D}$ (with the first and last word possibly incomplete) in no more than $(2K_0+1)^{1+N/n}$ ways. Indeed, fixing the partition means fixing the partition points, which are in distance at least $n$ and at most $n+2K_0$ from each other. That is, at most $1+N/n$ times we have to make a decision, and this decision can be made in $2K_0+1$ ways, hence the formula.

Thus, the noninjectivity can decrease the measure-theoretic entropy with respect to $\mu$ by at most

\[
\lim_{N\to\infty} \frac 1N \log (2K_0+1)^{1+N/n} = \frac 1n \log (2K_0+1)
\]
(compared with the 'ideal' situation given by the Abramov formula \cite{Ab}), which ends the proof of \eqref{entropies-relation}.
\color{black}

Similarly,  the proof of \eqref{LE-relations} follows from Remark \ref{LE exists}.
\end{proof}
\color{black}
\begin{thm}[{\cite[Theorem 4.8]{Moh22-entropy}}]\label{continuity_potential}Assume that $(A_1, \ldots, A_k) \in \glr^k$ generates a one-step cocycle $\A: \Si \to \glr.$  Suppose that $\mathcal{A}:\Si \to \glr$ is a typical cocycle. Then, 
\[\lim _{n \rightarrow \infty} \frac{1}{n} P_{n, \mathcal{D}}(\langle q, \Psi(\mathcal{B}) \rangle )=P(\log \psi^{q}(\mathcal{A})),\]
for each $q\in\R^d$ (uniformly on compact subsets).
\end{thm}
\medskip

Now, we are in a position to prove Theorem \ref{vair-prin-for the generalized sing}.

\begin{proof}[Proof of Theorem \ref{vair-prin-for the generalized sing}]

\color{black}

Since $\{\langle q, \Psi(\mathcal{B}^m) \rangle \}_{m \in \N}$ is almost additive for each $q \in \R^d$ (see Lemma \ref{psi is almost additive}), we have a varitional principle for $\langle q, \Psi(\mathcal{B}^m) \rangle$, i.e.,
\begin{equation}\label{var-prin-for-almost addititve}
P_{n, \mathcal{D}}(\langle q, \Psi(\mathcal{B}) \rangle)=\sup\bigg\{ h_{\mu}(f)+\lim_{m \to \infty} \frac{1}{m} \int  \langle q, \Psi(\B^m(x)) \rangle d\mu(x): \mu \in \M((\L_{n}^{\D})^{\N}, f)\bigg\}.
\end{equation}

For any $q \in \mathbb{R}^d$,  by \cite[Theorem 10.1.9]{Barreira}, the supremum is attained by a unique ergodic equilibrium state $\mu_{n}^q \in \M((\L_{n}^{\D})^{\N}, f)$, which has Gibbs property. Therefore,
\begin{equation}\label{var-prin-for-almost addititve1}
P_{n, \mathcal{D}}(\langle q, \Psi(\mathcal{B}) \rangle)=h_{\mu_{n}^q}\left(f \right)+\lim_{m \to \infty} \frac{1}{m} \int  \langle q, \Psi(\B^m(x)) \rangle d\mu_{n}^q(x).
\end{equation}

Note that for  $\mu_{n}^q \in \M((\L_{n}^{\D})^{\N}, f)$, there is $\nu_{n}^q \in \M(\Si, T)$ such that  \begin{equation}\label{Lyapunov-exponent-relation}
\lim_{m \to \infty} \frac{1}{m} \int \langle q, \Psi(\B^m(x)) \rangle d\mu_{n}^q(x) \leq (n+2K_{0})\lim_{m\to \infty} \frac{1}{m} \int \log \psi^{q}(\A^m(x)) d\nu_{n}^q(x) 
\end{equation}
and \begin{equation}\label{relation between entropies}
 h_{\mu_{n}^q}(f) \leq (n+2K_0)h_{\nu_{n}^q}(T)+\frac {n+2K_0}n \log (2K_0+1),
\end{equation} 
by Proposition \ref{relation between entropies and LE}.

By \eqref{var-prin-for-almost addititve}, \eqref{var-prin-for-almost addititve1}, \eqref{Lyapunov-exponent-relation}, \eqref{relation between entropies} and Theorem \ref{upper-bound of var-pri},
$$\begin{aligned}
& \frac{1}{n}P_{n, \mathcal{D}}(\langle q, \Psi(\B) \rangle) \stackrel{\eqref{var-prin-for-almost addititve}}{=}\frac{1}{n}\sup\bigg\{ h_{\mu}(f)+\lim_{m \to \infty} \frac{1}{m} \int \langle q, \Psi(\B^m(x)) \rangle d\mu: \mu \in \M((\mathcal{L}_{n}^{\D})^{\N}, f)\bigg\}\\
& \stackrel{\eqref{var-prin-for-almost addititve1}}{=}\frac{1}{n} \left( h_{\mu_{n}^q}\left(f \right)+\lim_{m \to \infty} \frac{1}{m} \int  \langle q, \Psi(\B^m(x)) \rangle d\mu_{n}^q(x) \right)\\
&\stackrel{\eqref{relation between entropies}\text{ and } \eqref{Lyapunov-exponent-relation}}{\leq} \frac{n+2K_0}{n} \left(h_{\nu_{n}^{q}}(T)+\lim_{m\to \infty} \frac{1}{m} \int \log \psi^{q}(\A^m(x)) d\nu_{n}^{q}(x) \right)\\
&+\frac {n+2K_0}{n^2} \log (2K_0+1) \\
& \leq \frac{n+2K_0}{n} \sup\bigg\{ h_{\mu}(T)+\lim_{m\to \infty} \frac{1}{m} \int \log \psi^{q}(\A^m(x)) d\mu(x): \mu \in \M(\Si, T) \bigg\}\\
&+\frac {n+2K_0}{ n^2} \log (2K_0+1)\\
&\stackrel{\text{Theorem }  \ref{upper-bound of var-pri}}{\leq} \frac{n+2K_0}{n} P(\log \psi^{q}(\A))+\frac {n+2K_0}{n^2} \log (2K_0+1).\\
\end{aligned}
$$

\color{black}

The statement follows from Theorem \ref{continuity_potential} by letting $n \to \infty.$

\end{proof}
\section{Proof of Theorem \ref{main-thm}}\label{proof-of-the-main-thm}
In this section, we prove Theorem \ref{main-thm}.
As an application of Legendre transform, we have the following theorem:
\begin{thm}\label{FH-cor}
Assume that $S$ is a non-empty, convex set in $\R^{d}$ and let $g:S\rightarrow \R$ be a concave function. Set
\[Z(x)=\sup\{g(a) + \langle a, x \rangle : a\in S \}, \hspace{0.3cm}x\in \R^{d}\]
and
\[G(a)=\inf\{Z(x) - \langle a, x \rangle : x\in \R^{d} \}, \hspace{0.3cm}a\in S.\]
Then $G(a)=g(a)$ for $a \in \text{ri}(S)$.
\end{thm}
\begin{proof}
See, e.g., \cite[Corollary  2.5]{FH}.
\end{proof}

Now, we prove Theorem \ref{main-thm}:

\begin{proof}[Proof of Theorem \ref{main-thm}]
 $\Omega$ is non-empty and convex. We define $g:\Omega \to \R$ by
\[g(\vec{\alpha})=\sup\bigg\{h_{\mu}(T): \mu \in \mathcal{M}(\Si, T), \quad (\chi_{1}(\mu, \mathcal{A}), \ldots, \chi_{d}(\mu, \mathcal{A}))=\vec{\alpha}\bigg\}.\]

Since the entropy map $\mu \mapsto h_\mu(T)$ is concave and the Lyapunov exponent map $\mu \mapsto (\chi_1(\mu,\mathcal{A}),\ldots,\chi_d(\mu,\mathcal{A}))$ is affine, convex combinations of invariant measures preserve the constraint and do not decrease entropy. Taking the supremum over such measures therefore shows that $g$ is a real-valued concave function on $\Omega$. \color{black}
We define
\[Z(x):=\sup\bigg\{g(\vec{\alpha})+ \langle \vec{\alpha}, x \rangle : \vec{\alpha}\in \Omega\bigg\}, \quad \forall x \in \R^{d}.\]

Let $S:\R^{d} \to \R \cup \{+\infty\}$ be the function which agrees with $-g$ on $\Omega$ but is $+\infty$ everywhere else. Then, $S$ is convex  and has $\Omega$ as its effective domain, i.e. $\Omega=\{x, S(x)< \infty \}.$ By the definition of $Z$, $Z$ is equal to the conjugate function of $S$, so $Z$ is a convex function on $\R^{d}$ by the Legendre transform property (see \cite[Subsection 3.2]{Moh22-Lyapunov}).

We have
\[g(\vec{\alpha})=\inf_{x\in \R^{d}}\bigg\{Z(x)- \langle \vec{\alpha}, x \rangle  \bigg\},\] 
for all $\vec{\alpha}\in \text{ri}(\Omega)$, by Theorem \ref{FH-cor}.

By Lemma \ref{existence-of-top-pre},  $P(\log \psi^q(\A))$ is a convex function of $q$ in $\R^{d}$. Then, by Theorem \ref{vair-prin-for the generalized sing}  (variational principle),
\[Z(q)=P(\log \psi^{q}(\A)) \hspace{0.3cm} \forall q\in \R^d .\]
 Thus,
\[
\begin{aligned}
&\sup\bigg\{h_{\mu}(T) : \mu \in \mathcal{M}(\Si, T), \chi_{i}(\mu, \mathcal{A})=\alpha_{i} \text{ for }i=1,2, \ldots, d \bigg\}=\\
&\inf_{q\in \R^d} \bigg\{P(\log \psi^{q}(\A))-\langle q, \vec{\alpha} \rangle \bigg\},
\end{aligned}\]
for $\vec{\alpha} \in \text{ri}(\Omega).$

\end{proof}

\section{Remarks}

Assume that $(A_1, \ldots, A_k)\in \glr^k$ generates a one-step cocycle $\A:\Si \to \glr$. We consider the potential $\mathcal{G}_t:=\{t \log \|\A^n\|\}_{n\in \N}$, which is sub-additive when $t\geq 0$ and superadditive when $t<0$. We are going to discuss the existence of the equilibrium measure for $\mathcal{G}_t$. In the sub-additive case, the supremum of the variational principle \eqref{varitional-principle} is attained as the measure-theoretic entropy and the Lyapunov exponent  are upper-semi continuous, so there exists an equilibrium measure for $\mathcal{G}_t$. In contrast, when $t < 0$, the functional
$
\mu \mapsto t \, \chi_1(\mu, \mathcal{A})
$
is lower semicontinuous, and equilibrium measures may fail to exist. We refer to \cite{mohammadpour-Quas} for further details.

Finally, we observe that $\mathcal{G}_t$ is a special case of the potential
$
\{ \log \psi^q(\mathcal{A}^n) \}_{n \in \mathbb{N}}
$
with $q = (t,0,\ldots,0)$. This naturally leads to the following question:

\begin{question}
Assume that $(A_1, \ldots, A_k)\in \mathrm{GL}(d,\mathbb{R})^k$ generates a one-step cocycle $\mathcal{A}:\Sigma \to \mathrm{GL}(d,\mathbb{R})$. Suppose that $\mathcal{A}$ is a typical cocycle. Is the supremum in the variational principle in Theorem~\ref{vair-prin-for the generalized sing} attained?
\end{question}

The answer to this question determines whether an equilibrium measure for the potential $\log \psi^{q}(\mathcal{A})$ exists.

\subsection*{Acknowledgements.}
The author would like to thank Michal Rams for
helpful discussion and helping with Proposition \ref{relation between entropies and LE}. He
also thanks the anonymous referee for their valuable corrections and suggestions. This work was supported by the Knut and Alice Wallenberg Foundation and the Swedish Research Council grant 104651320.

\bibliographystyle{acm}
\bibliography{variational-typical}
\end{document}